\documentclass[11pt,a4paper]{amsart}
\usepackage{amsmath,amssymb,amsthm,amscd,verbatim,enumerate}
\usepackage{graphicx}
\usepackage[colorlinks=true,citecolor=black,linkcolor=black,urlcolor=blue]{hyperref}
\usepackage[lmargin=31mm,rmargin=31mm,bmargin=31mm,tmargin=31mm]{geometry}

\renewcommand{\baselinestretch}{1.1}
\allowdisplaybreaks

\renewcommand{\l}{\ell}

\theoremstyle{plain}
\newtheorem{theorem}{Theorem}[section]
\newtheorem{lemma}[theorem]{Lemma}
\newtheorem{corollary}[theorem]{Corollary}
\newtheorem{proposition}[theorem]{Proposition}

\newtheorem{claim}[theorem]{Claim}

\theoremstyle{definition}

\renewcommand{\leq}{\leqslant}

\renewcommand{\geq}{\geqslant}

\newcommand{\N}{\mathbb{N}}
\newcommand{\DEF}{\sl}    

\newcommand{\eps}{\varepsilon}  
\newcommand{\mc}{\mathcal}
\newcommand{\F}{\mathcal{F}}
\newcommand{\M}{\mathcal{M}}
\newcommand{\PP}{\mathcal{P}}
\newcommand{\tsep}[1]{pw-$#1$-separation}

\begin{document}

\title{Excluded Forest Minors and the {E}rd\H{o}s--{P}\'osa Property}

\author{Samuel Fiorini}
\address{\newline  D\'epartement de Math\'ematique
\newline Universit\'e Libre de Bruxelles
\newline Brussels, Belgium}
\email{sfiorini@ulb.ac.be}

\author{Gwena\"el Joret}
\address{\newline  D\'epartement d'Informatique 
\newline Universit\'e Libre de Bruxelles
\newline Brussels, Belgium}
\email{gjoret@ulb.ac.be}

\author{David~R.~Wood}
\address{\newline School of Mathematical Sciences
\newline Monash University
\newline Melbourne, Australia}
\email{david.wood@monash.edu}

\subjclass[2000]{05C83 Graph minors, 05C35 Extremal problems}

\thanks{Research of David Wood is supported by the Australian Research Council.}

\date{\today}
\sloppy
\maketitle

\begin{abstract}
A classical result of Robertson and Seymour states that the set of graphs containing a fixed planar graph $H$ as a minor has the so-called Erd\H{o}s-P\'osa property; namely, there exists a function $f$ depending only on $H$ such that, for every graph $G$ and every positive integer $k$, the graph $G$ has $k$ vertex-disjoint subgraphs each containing $H$ as a minor, or there exists a subset $X$ of vertices of $G$ with $|X| \leq f(k)$ such that $G - X$ has no $H$-minor (see {\it N.~Robertson and P.~D.~Seymour. Graph minors. V. Excluding a planar graph. J. Combin. Theory Ser.~B, 41(1):92--114, 1986}). While the best function $f$ currently known is exponential in $k$, a $O(k \log k)$ bound is known in the special case where $H$ is a forest. This is a consequence of a theorem of Bienstock, Robertson, Seymour, and Thomas on the pathwidth of graphs with an excluded forest-minor. In this paper we show that the function $f$ can be taken to be linear when $H$ is a forest. This is best possible in the sense that no linear bound is possible if $H$ has a cycle. 
\end{abstract}

\section{Introduction}

Let $\F$ be a finite set of graphs, which we will typically think of as
a set of excluded (or forbidden) minors. 
Given a graph $G$, an {\DEF $\F$-packing} (or simply {\DEF packing}) in $G$
is a collection of vertex-disjoint subgraphs of $G$ each containing a member of 
$\F$ as a minor. The maximum size of such a collection is denoted $\nu_{\F}(G)$.
A dual notion is that of an {\DEF $\F$-transversal} (or simply {\DEF transversal}) 
of $G$, which is defined as a subset $X$ of vertices of $G$ such that $G-X$
contains no member of $\F$ as a minor. The minimum size of such a set $X$ is denoted 
$\tau_{\F}(G)$.

Thus if $\mc G$ is a proper minor-closed class of graphs and $\F$ is the
corresponding set of minimal forbidden minors (which is finite by the graph minor theorem~\cite{RS04}), 
then $\tau_{\F}(G)$ is the minimum number of vertices to remove from $G$ to obtain 
a graph in the class $\mc G$, while $\nu_{\F}(G)$ is the maximum number
of vertex-disjoint forbidden minors in $G$. 

These two graph invariants generalize  
some classical invariants in graph theory 
and combinatorial optimization: If $\F = \{K_{2}\}$
then it is easily seen that $\nu_{\F}(G)$ is the maximum size of a matching in $G$,
while $\tau_{\F}(G)$ is the minimum size of a vertex cover in $G$. 
Similarly, if $\F = \{K_{3}\}$, then $\nu_{\F}(G)$ is the maximum size of
a cycle packing in $G$, and $\tau_{\F}(G)$ is the minimum size of a cycle transversal (also known as a feedback vertex set) of $G$. 

Clearly $\tau_{\F}(G) \geq \nu_{\F}(G)$ for every $\F$ and every graph $G$. 
If at least one graph in $\F$ is planar then 
the two parameters are ``tied'' to each other in the following sense: 
There exists a function $f$ depending only on $\F$ such that 
$\tau_{\F}(G) \leq f(\nu_{\F}(G))$  for every graph $G$. 
This was shown by Robertson and Seymour~\cite{RS_five}\footnote{It should be 
noted that this result is proved in~\cite{RS_five}  
in the special case where $\F$ consists of a single planar graph $H$. 
However the general case follows by a straightforward
modification of the proof of (8.8) in~\cite{RS_five}.} and 
is a consequence of their well-known excluded grid theorem~\cite{RS_five}. 
It is often referred to as
the {\DEF Erd\H{o}s--P\'osa property} of $\F$
(or more accurately, of the set of graphs contractible to some graph in $\F$), 
because {E}rd\H{o}s and P\'osa~\cite{EP} proved
the existence of such a function $f$ when $\F=\{K_{3}\}$. 

Robertson and Seymour's result is best possible in the sense
that no such function $f$ exists if no graph in $\F$ is planar 
(see~\cite{RS_five}). However, the function $f$ that follows
from their proof is huge (exponential). 
This is because, as mentioned above, the proof relies on 
their excluded grid theorem, namely for every $r \in \N$
there exists a minimum integer $g(r)$ such that every graph 
with no $r\times r$-grid minor has treewidth at most $g(r)$,     
and the current best upper bound on $g(r)$ is exponential: 
$g(r) \in 2^{O(r^{2}\log r)}$~\cite{KK12, LS12}. 
(As for lower bounds, only $g(r) \in \Omega(r^{2}\log r)$ is known~\cite{QuicklyPlanar}.)
 
The situation changes drastically if instead of simply assuming that some
graph in $\F$ is planar, we further require that $\F$ contains a forest $F$.  
Bienstock, Robertson,  Seymour, and Thomas~\cite{QuicklyForest} proved that
every graph with pathwidth at least $|F|-1$ contains $F$ as a minor. 
Using this result, one can derive without much difficulty that
$\tau_{\F}(G) \in O_{\F}(k \log k)$ where $k=\nu_{\F}(G)$; 
see Proposition~\ref{prop:log} in Section~\ref{sec:preliminaries}. 
As far as we are aware, this is the best bound that is currently known. 

In this paper we prove that a linear bound holds when $\F$ contains a forest:

\begin{theorem}
\label{thm:main}
Let $\F$ be a finite set of graphs containing at least one forest. 
Then there exists a computable constant $c = c(\F)$ such that $\tau_{\F}(G) \leq c \cdot \nu_{\F}(G)$ 
for every graph $G$.
\end{theorem}

If for some $t\in \N$, we let $\F$ be the (finite) set of minimal forbidden minors
for the class of graphs with pathwidth at most $t$, 
then $\F$ contains at least one tree. This follows from the fact that trees have unbounded pathwidth 
(for instance, the complete binary tree of height $h$ has pathwidth $h$). 
Using that $\F$ is computable~\cite{AGK08, L98}, we obtain the following corollary from Theorem~\ref{thm:main}.

\begin{corollary}
\label{cor:main}
For every $t\in \N$ there exists a computable constant $c=c(t)$ such that, 
for every graph $G$ and every $k\in \N$, either $G$ contains
$k$ vertex-disjoint subgraphs each with pathwidth at least $t + 1$, or $G$ has a vertex subset $X$ 
of size at most $c\cdot k$ such that $G-X$ has pathwidth at most $t$. 
\end{corollary}

Interestingly, Corollary~\ref{cor:main} becomes false  
if we replace pathwidth by treewidth, even for $t=1$: 
A graph has treewidth at least $2$ if and only if it contains a cycle, and   
there are graphs with no $k$ vertex-disjoint cycles such that
every cycle transversal has size $\Omega(k\log k)$ (see~\cite{EP}). 
More generally, for every fixed set $\F$ containing a planar graph but no forest, there are graphs $G$ with $\nu_{\F}(G) = k$
and $\tau_{\F}(G) \in \Omega_{\F}(k \log k)$; 
this follows from the existence of 
$n$-vertex graphs $G$ with treewidth $\Omega(n)$ and girth $\Omega(\log n)$.\footnote{Indeed, by~\cite{RS_five} graphs containing no member of $\F$ as a minor have 
treewidth at most $c$ for some constant $c = c(\F)$ (since $\F$ contains a planar graph). 
Thus $G- X$ has treewidth at most $c$ for every $\F$-transversal $X$ of $G$, and hence
$\tau_{\F}(G) = \Omega(n)$. On the other hand, every subgraph of $G$ containing 
a member of $\F$ as a minor has a cycle, and thus has $\Omega(\log n)$ vertices, implying $\nu_{\F}(G) = O(n / \log n)$. 

We remark that to obtain $n$-vertex graphs $G$ with treewidth $\Omega(n)$ and girth $\Omega(\log n)$ it suffices to consider $d$-regular $n$-vertex 
expanders of girth $\Omega(\log n)$ for some fixed $d$, such as the 
Ramanujan graphs constructed by Lubotzky, Phillips, and Sarnak~\cite{LPS88} for 
instance. 
The fact that the treewidth is $\Omega(n)$ is a direct consequence of their positive 
vertex-expansion (or see~\cite [Proposition~1]{GM09}).  
} 
In this sense Theorem~\ref{thm:main} is best possible. 

Our proof of Theorem~\ref{thm:main} can very briefly be described as follows: 
First suppose for simplicity that $\F = \{T\}$ with $T$ a tree; the general case can 
be reduced to this one without much difficulty. 
We show that every graph $G$ either has a model of $T$ of constant size, or a reduction 
operation can be applied to $G$, producing a graph $G'$ with 
$\nu_{\F}(G) = \nu_{\F}(G')$ and 
$\tau_{\F}(G) = \tau_{\F}(G')$ which is smaller than $G$. 
In the first case we remove all vertices of the model and apply induction, while 
in the second case we are done by applying induction on $G'$. 
The reduction operation is described in Section~\ref{sec:reduction}. 
Section~\ref{sec:proof} is devoted to the proof 
that every {\em reduced} graph contains a constant-size $T$-model. 
Section~\ref{sec:algorithms} describes a number of algorithmic applications 
of (the proof of) Theorem~\ref{thm:main}. 

We conclude the introduction by mentioning two 
recent related results. 
First, Diestel, Kawarabayashi, and Wollan~\cite{DKW12} 
studied the $\F=\{K_{t}\}$ case and showed that, while 
$\tau_{\F}(G)$ cannot be bounded from above by a function of $\nu_{\F}(G)$  
when $t \geq 5$, such a function exists 
for graphs with high enough vertex-connectivity  
compared to 
$\nu_{\F}(G)$ and $t$: There exists a function $f$ such that, 
for every $t, k \geq 0$, every $(k(t - 3) + 14t + 14)$-connected graph $G$ 
either has an $\F$-packing of size $k$ or an $\F$-transversal of size at most $f(k, t)$. 
Moreover, for fixed $t$ the connectivity requirement is best possible up to an additive constant.

Second, Fomin, Saurabh, and Thilikos~\cite{ErdosPosaMinorClosed} 
considered restricting the graphs $G$ under consideration to belong to some fixed
proper minor-closed family of graphs $\mc G$, and 
showed that if $\F=\{H\}$ with $H$ a connected planar graph, then
there exists a constant $c=c(\F, \mc G)$ such that
$\tau_{\F}(G) \leq c \cdot \nu_{F}(G)$ for every $G \in \mc G$. 
Note that, as mentioned above, there is no such linear bound if we do not
impose $G \in \mc G$.

\section{Definitions and Preliminaries}
\label{sec:preliminaries}

All graphs in this paper are finite, simple, and undirected, unless otherwise
stated.  Let $V(G)$ and $E(G)$ denote the vertex and edge sets of a
graph $G$. Let $[i,j]:=\{i, i + 1, \dots, j\}$. 

A {\DEF tree decomposition} of a graph $G$ is a pair
$(T, \{B_{x} : x \in V(T)\})$ where $T$ is a tree, and $\{B_{x}
: x \in V(T)\}$ is a family of subsets of $V(G)$ (called {\DEF
 bags}) such that
 \begin{itemize}
 \item $\bigcup_{x \in V(T)} B_{x} = V(G)$;
 \item for every edge $uv \in E(G)$, there exists $x \in V(T)$ with
   $u,v \in B_{x}$, and
 \item for every vertex $u\in V(G)$, the set $\{x\in V(T) : u \in
   B_{x}\}$ induces a subtree of $T$.
 \end{itemize}
 The {\DEF width} of a tree decomposition $(T, \{B_{x} : x \in
 V(T)\})$ is $\max \{ |B_{x}| - 1 : x \in V(T)\}$.  The {\DEF
   treewidth} of $G$ is the minimum width among all tree
 decompositions of $G$. See \cite{DiestelBook} for an introduction to
 the theory of treewidth.  
 A {\DEF path decomposition} of $G$ is a tree decomposition of $G$ where
 the underlying tree $T$ is a path. 
 We will denote such a decomposition simply by the sequence 
 $B_{1}, \dots, B_{p}$ of its bags (in order). 
 The {\DEF pathwidth} of $G$ is the minimum width of a path decomposition of $G$.

A graph $H$ is a {\DEF minor} of a graph $G$ if $H$ can be obtained
from a subgraph of $G$ by contracting edges. (Note that, since we only
consider simple graphs, loops and parallel edges created during an
edge contraction are deleted.)  An {\DEF $H$-model} in
$G$ is a collection $M=\{S_{x}: x\in V(H)\}$ of vertex-disjoint
connected subgraphs of $G$ (called {\DEF branch sets}) such that, for
every edge $xy \in E(H)$, some edge in $G$ joins a vertex in $S_{x}$
to a vertex in $S_{y}$.  
The {\DEF vertex set} of $M$ is $V(M) := \cup\{V(S_{x}): x\in V(H)\}$. 
Clearly, $H$ is a minor of $G$ if and only if
$G$ contains an $H$-model.

The following result of Bienstock, Robertson,  Seymour, and Thomas~\cite{QuicklyForest} 
will be routinely used in our proofs (see~\cite{DiestelShort} for a short proof).

\begin{theorem}[\cite{QuicklyForest}]
\label{thm:quickly_forest}
Every graph with pathwidth at least $t-1$ contains every forest on $t$ vertices as a minor.  
\end{theorem}
 
Note that the bound in Theorem~\ref{thm:quickly_forest} is best possible. 

We proceed with a few simple lemmas that will be used in our proof of 
Theorem~\ref{thm:main}. 
The following lemma is a special case of (8.6) in~\cite{RS_five}. 

\begin{lemma}[\cite{RS_five}]
\label{lem:intervals}
Let $P$ be a path and let $\PP_{1}, \dots, \PP_{m}$ be families of subpaths of $P$. 
Let $x_{1}, \dots, x_{m} \geq 0$ be integers, and let $k := x_{1} + \cdots + x_{m}$. 
Suppose that for each $i\in [1,m]$ there are $k$ members of $\PP_{i}$ that are 
pairwise vertex-disjoint. 
Then there exist $x_{i}$ members $P^{i}_{1}, \dots, P^{i}_{x_{i}}$ in $\PP_{i}$ for each $i\in [1,m]$
such that $P^{1}_{1}, \dots, P^{1}_{x_{1}}, P^{2}_{1}, \dots, 
P^{2}_{x_{2}}, \dots, P^{m}_{1}, \dots, P^{m}_{x_{m}}$ are all pairwise vertex-disjoint. 
\end{lemma}

The next lemma can be derived 
from the proof of (8.8) in~\cite{RS_five}. We provide a proof
for completeness. 

\begin{lemma}
\label{lem:bounded_pathwidth}
Let $\F$ be a set of $q \geq 1$ graphs, each with at most $r$ components,  
and let $t \geq 1$. 
Then $\tau_{\F}(G) \leq 2qrt\cdot \nu_{\F}(G)$ for every graph $G$ with 
pathwidth strictly less than $t$. 
\end{lemma}
\begin{proof} 
The claim is trivially true if $\nu_{\F}(G)=0$, so assume $\nu_{\F}(G) \geq 1$. 
Let $\F=\{H_{1}, \dots, H_{q}\}$. For each $i \in [1,q]$, let $H_{i, 1}, \dots, H_{i, c_{i}}$ denote
the components of $H_{i}$, where $c_{i} \leq r$. 
Let $s:= \nu_{\F}(G) + 1$. 
Let $B_{1}, \dots, B_{p}$ denote a path decomposition of $G$ of width at most $t-1$. 
Let $P$ denote the path on $p$ vertices with vertex set $\{B_{1}, \dots, B_{p}\}$ (in order). 

For each $i\in [1,q]$ and $j \in [1, c_{i}]$, let $\M_{i,j}$ be the set of $H_{i,j}$-models 
in $G$. Observe that for each $M \in \M_{i,j}$ we have that $V(M)$ induces a connected subgraph of
$G$, since $H_{i,j}$ is connected; thus vertices in $V(M)$ appear in consecutive bags of the path decomposition, and hence $M$ defines a corresponding subpath $P_{M}$ of $P$.  
Notice that a sufficient condition for two models $M \in \M_{i,j}$ and $M' \in \M_{i',j'}$ 
being vertex-disjoint is that $P_{M}$ and $P_{M'}$ are vertex-disjoint; this will be used below.  

First suppose that there exists an index $i\in [1,q]$ such that for each $j \in [1, c_{i}]$, 
there are $sc_{i}$ pairwise vertex-disjoint paths in $\{P_{M}: M \in \M_{i,j}\}$.
Then by Lemma~\ref{lem:intervals} (with $m=c_{i}$, $k=sc_{i}$ and $x_{1}=\cdots=x_{m}=s$), 
one can find $s$ $H_{i,j}$-models $M_{i,j}^{1}, \dots,  M_{i,j}^{s}$ in $\M_{i,j}$ 
for each $j \in [1, c_{i}]$ such that 
$M_{i,1}^{1}, \dots,  M_{i,1}^{s}, M_{i,2}^{1}, \dots,  M_{i,2}^{s}, \dots,
M_{i,c_{i}}^{1}, \dots,  M_{i,c_{i}}^{s}$ are all pairwise vertex-disjoint. 
In particular, $G$ has $s$ vertex-disjoint $H_{i}$-models, which implies 
$\nu_{\F}(G) \geq s$, a contradiction. Hence there is no such index $i$. 

For each $i\in [1,q]$, consider an index $j \in [1, c_{i}]$ such that
$\{P_{M}: M \in \M_{i,j}\}$ has no $sc_{i}$ pairwise vertex-disjoint paths. 
Then, as is well-known, one can find a subset of at most (in fact, strictly less than) 
$sc_{i}$ vertices of $P$
that meet all the paths in $\{P_{M}: M \in \M_{i,j}\}$; let $X_{i}$ denote the union 
of the bags in the path decomposition corresponding to such a subset (thus $|X_{i}| \leq sc_{i}t$).  
Observe that $G-X_{i}$ has no $H_{i,j}$-minor, and hence no $H_{i}$-minor. 
Therefore, $X_{1}\cup \cdots \cup X_{q}$ is an $\F$-transversal of $G$ of size
at most $sc_{1}t + \cdots + sc_{q}t \leq sqrt = qrt\cdot\nu_{\F}(G) + qrt \leq 2qrt\cdot\nu_{\F}(G)$. 
\end{proof}

If $\F$ is a set consisting of a single graph $H$, we simple write 
$\nu_{H}(G)$ and $\tau_{H}(G)$ for the invariants $\nu_{\F}(G)$ and $\tau_{\F}(G)$, respectively. 

Suppose that $\F$ is a finite set of graphs containing a forest $F$, and let $T$
be a tree on $|F|$ vertices containing $F$. 
The next lemma allows us to reduce the task of finding a small $\F$-transversal of $G$ 
to that of finding a small $\{T\}$-transversal of $G$. 
This will be used in the proof of Theorem~\ref{thm:main}. 

\begin{lemma}
\label{lem:F_to_tree}
Let $\F$ be a set of $q \geq 1$ graphs, each with at most $r$ components,  
containing a forest $F$ on $t$ vertices. Let $T$ be a tree on $t$ vertices
with $F \subseteq T$. 
Then $$\nu_{T}(G) \leq \nu_{\F}(G) \leq 
\tau_{\F}(G) \leq \tau_{T}(G) +  2qrt\cdot \nu_{\F}(G)$$ for every graph $G$. 
\end{lemma}
\begin{proof}
The inequalities $\nu_{T}(G) \leq \nu_{\F}(G) \leq \tau_{\F}(G)$ are obvious; 
let us show that $\tau_{\F}(G) \leq \tau_{T}(G) +  2qrt\cdot \nu_{\F}(G)$. 
Let $X$ be a $\{T\}$-transversal of $G$ with $|X| = \tau_{T}(G)$. 
Then $G-X$ has no $T$-minor, and hence has pathwidth at most $t-2$ by 
Theorem~\ref{thm:quickly_forest}. 
By Lemma~\ref{lem:bounded_pathwidth}, $G-X$ has an $\F$-transversal $Y$
with $|Y| \leq 2qrt \cdot \nu_{\F}(G-X) \leq 2qrt \cdot \nu_{\F}(G)$. 
It follows that $X\cup Y$ is an $\F$-transversal of $G$ of size at most 
$\tau_{T}(G) +  2qrt \cdot \nu_{\F}(G)$. 
\end{proof}

As mentioned in the introduction, it is not difficult to derive from 
Theorem~\ref{thm:quickly_forest}   
that, when $\F$ is a finite set of graphs containing a forest, 
$\tau_{\F}(G) \in O_{\F}(k \log k)$ where $k= \nu_{\F}(G)$. 
We conclude this section with a proof of this statement, for completeness.  

First we consider the case where $\F$ consists of a single tree.

\begin{lemma}
\label{lem:klogk_tree}
Let $T$ be a tree on $t$ vertices. 
Then $\tau_{T}(G) \leq 3(t+1) k\log_2 ((t+1)k) - t$ 
for every graph $G$ with $\nu_{T}(G)=k$.
\end{lemma}
\begin{proof}
The proof is by induction on $k$.
If the pathwidth of $G$ is at least $t(k+1) - 1$, then $G$
contains every forest on at most $t(k+1)$ vertices as a minor, and thus contains 
$k+1$ disjoint copies of $T$ as a minor. This implies $\nu_{T}(G) \geq k + 1$, a contradiction. 
Hence, $G$ has pathwidth at most $t(k+1) - 2$. 

Let $B_{1}, \dots, B_{p}$ be a path decomposition of $G$ of width 
at most $t(k+1) - 2$. By modifying the decomposition if necessary, we may assume
that $|B_{i} \bigtriangleup B_{i+1}| = 1$ for each $i \in [1, p-1]$. 

For each $i\in [1,p]$, let $L_{i} := G[(B_{1} \cup \cdots \cup B_{i-1}) - B_{i}]$ and
$\l_{i} := \nu_{T}(L_{i})$, and similarly  
let $R_{i} := G[(B_{i+1} \cup \cdots \cup B_{p}) - B_{i}]$ and 
$r_{i} := \nu_{T}(R_{i})$. Observe that $\l_{i} + r_{i} \leq k$, 
since $L_{i}$ and $R_{i}$ are vertex-disjoint subgraphs of $G$. 
Also, since $|B_{i} \bigtriangleup B_{i+1}| = 1$, we have
$\l_{i+1} \in \{ \l_{i}, \l_{i} + 1\}$ for each $i \in [1, p-1]$. 
Thus there exists an index $j\in [1,p]$ such that 
$\l_{j} \leq \left\lceil \frac{k}{2} \right\rceil$
and $r_{j} \leq \left\lfloor \frac{k}{2} \right\rfloor$. 

For the base case of the induction, namely $k=1$, we have $\l_{j} \leq 1$ and $r_{j}= 0$. 
If $\l_{j} = 0$ then $Z:=B_{j}$ is an $\{T\}$-transversal of $G$
of size at most $t(k+1) - 1$. 
Otherwise, $\l_{j} = 1$ and $j > 1$ (since $\l_{1} = 0$).  
We may further assume that $j$ is chosen so that $\l_{j - 1}= 0$.
Then $Z := B_{j-1} \cup B_{j}$ is an $\{T\}$-transversal of $G$
of size at most $t(k+1)$ (since $|B_{j-1} \bigtriangleup B_{j}| = 1$). 
Thus in each case the transversal $Z$ has size at most
$t(k+1) = 2t \leq 3(t+1)\log_2 (t+1) - t$.

For the inductive step, assume $k \geq 2$. 
By induction there are $\{T\}$-transversals $X$ and $Y$ of $L_{j}$ and $R_{j}$, respectively, 
such that $|X| \leq 3(t+1) \left\lceil \frac{k}{2} \right\rceil \log_2 \left((t+1) \left\lceil \frac{k}{2} \right\rceil\right) - t$
and $|Y| \leq 3(t+1) \left\lfloor \frac{k}{2} \right\rfloor \log_2 \left((t+1) \left\lfloor \frac{k}{2} \right\rfloor\right) - t$. 
Then $X \cup Y \cup B_{j}$ is an $\{T\}$-transversal of $G$
of size at most
\begin{align*}
&\quad\; 3(t+1) \left\lceil \tfrac{k}{2} \right\rceil \log_2 \left((t+1) \left\lceil \tfrac{k}{2} \right\rceil\right) - t
+ 3(t+1) \left\lfloor \tfrac{k}{2} \right\rfloor \log_2 \left((t+1) \left\lfloor \tfrac{k}{2} \right\rfloor\right) - t 
+ t(k+1) -1 \\
&\leq 3(t+1)k \log_{2} (t+1) 
+ 3(t+1)k \log_{2} \left(\tfrac{k+1}{2} \right)
+ tk - t -1 \\
&= 3(t+1)k \log_{2} (t+1) 
+ 3(t+1)k \log_{2} (k+1)
- 2tk - 3k - t -1 \\
&\leq 3(t+1)k \log_{2} (t+1) 
+ 3(t+1)k \log_{2} \left( \tfrac{3k}{2}\right)
- 2tk - 3k - t -1 \\
& \leq 3(t+1)k \log_{2}((t+1)k) - t, 
\end{align*}
as desired.
\end{proof}

Combining Lemmas~\ref{lem:F_to_tree} and~\ref{lem:klogk_tree}, 
we obtain the aforementioned $O_{\F}(k \log k)$ bound.   

\begin{proposition}
\label{prop:log}
Let $\F$ be a set of $q \geq 1$ graphs, each with at most $r$ components,  
containing a forest on $t$ vertices. 
Then $\tau_{\F}(G) \leq 3(t+1) k\log_2 ((t+1)k) + 2qrt k - t$ 
for every graph $G$ with $\nu_{\F}(G)=k$.
\end{proposition}

\section{A Reduction Operation}
\label{sec:reduction}

First we need to introduce some definitions. 
A {\DEF rooted graph} is a pair $(G, R)$ where $G$ is a graph and 
$R = (v_1, \dots, v_k)$ is a (possibly empty) ordered subset of vertices of $G$ called {\DEF roots}. 
The notions of minors and models generalize in a natural way to rooted graphs: 
A rooted graph $(H, R')$ with $R'=(w_1, \dots, w_\l)$ is a {\DEF minor}
of $(G, R)$ if $k=\l$ and there exists a collection $\{V_u: u \in V(H)\}$ of
disjoint vertex subsets of $G$ (the {\DEF branch sets}), each inducing a connected subgraph in $G$, such that $v_i \in V_{w_i}$ for each $i \in [1,k]$, and there exists
an edge between $V_{u}$ and $V_{v}$ in $G$ for every $uv \in E(H)$. 
The collection $M:=\{V_u: u \in V(H)\}$
is called a {\DEF model} of $(H, R')$ in $(G, R)$.  We write $V(M)$ to denote the set
$\cup_{u \in V(H)} V_u$ of vertices of the model $M$.  

Isomorphism between rooted graphs is defined in the expected way (the $i$-th root of the first graph
is required to be mapped to the $i$-th root of the second). 
We write $R' \subseteq R$ for an ordered subset
of an ordered set $R$ (thus, the ordering of $R'$ is 
consistent with the ordering of $R$). 
Also, if $X$ is an (unordered) set then 
$R-X$ is ordered as in $R$. 

For our reduction operation we will need to keep track of rooted minors
of bounded sizes of a given rooted graph $(G,R)$.  
For an integer $q\geq1$, the {\DEF
 $q$-folio} of $(G, R)$ is the $2^{|R|}$-tuple which, for each set $R' \subseteq R$,  
 records the set of all minors $(H, R')$ of $(G, R')$ with $|H| \leq q$ 
 (keeping one member per isomorphism class). 
For integers $p \geq 0$ and
$q\geq 1$, the {\DEF $p$-deletion $q$-folio} of $(G, R)$ is the $(p+1)$-tuple which, 
for each $i\in [0,p]$, records for each $X \subseteq R$ with $|X| \leq i$ 
the set of all $q$-folios of $(G - (X \cup Y), R - X)$ taken over all subsets
$Y \subseteq V(G) - R$ with $|X| + |Y| = i$. 
(Informally speaking, we record the different $q$-folios obtained by deleting $i$ vertices 
in or outside the set of roots, for each $i\in [0,p]$.)

The following lemma on $p$-deletion $q$-folios is an 
adaptation of Lemma 2.2 in~\cite{GroheFPT}, its proof relies on 
standard monadic second-order logic techniques.  

\begin{lemma}
\label{lem:MSO}
Let $p, r, z \geq 0$ and $q \geq 1$. Then 
there exists a computable function $g(p,q,r,z)$ such that, for every rooted graph $(G, R)$ 
such that $G$ has treewidth at most $z$ and $|R|=r$, 
there exists a rooted graph $(G', R)$ with $|G'| \leq g(p,q,r, z)$
and such that $(G, R)$ and $(G', R)$ have the same 
$p$-deletion $q$-folio.
\end{lemma}
\begin{proof}
Fix an ordered set $R$ with $|R| = r$. 
Let $\mathcal{G}(R, q)$ denote the set of (non-isomorphic) rooted graphs $(G, R)$ 
with $|G| \leq q$. Clearly $|\mathcal{G}(R, q)|$ is bounded from above by a function
of $r$ and $q$. 

Let $\mathcal{S}(R, q)$ denote the set of vectors $s$ having one entry $s(R')$ per subset 
$R'\subseteq R$ and such that $s(R') \subseteq \mathcal{G}(R', q)$.
Again, $|\mathcal{S}(R, q)|$ is bounded from above by a function
of $r$ and $q$. 
By definition, for each rooted graph $(G, R)$ 
there is exactly one vector in $\mathcal{S}(R, q)$ that encodes 
the $q$-folio of $(G, R)$. (Note that, on the other hand, some vectors 
$s\in \mathcal{S}(R, q)$ might not correspond 
to the $q$-folio of any rooted graph $(G, R)$.)

Let $\mathcal{D}(R, q, p)$ denote the set of vectors $d=(D_{0}, \dots, D_{p})$ such that, 
for each $i\in [0,p]$, $D_{i}$ has one entry $D_{i}(X)$ for each $X \subseteq R$ with $|X| \leq i$, 
and $D_{i}(X) \subseteq \mathcal{S}(R - X, q)$ for each such set $X$. 
The size of $\mathcal{D}(R, q, p)$ is bounded from above by a function
of $r$, $q$, and $p$. 
By the definition of $\mathcal{D}(R, q, p)$, 
the $p$-deletion $q$-folio of a rooted graph $(G, R)$ 
is encoded by exactly one vector in $\mathcal{D}(R, q, p)$. 

It is well-known that the property that a fixed graph $H$ is a minor
of a graph $G$ can be expressed by a monadic second-order logic (MSO) sentence $\varphi_{H}$ 
(see Grohe~\cite{GroheSurvey} for an introduction to monadic second-order logic). 
Similarly, for each $s \in \mathcal{S}(R, q)$ one can write an MSO-sentence $\varphi_{s}$ 
such that a rooted graph $(G, R)$ satisfies $\varphi_{s}$ if and only if its $q$-folio equals $s$. 
More generally, for each $d\in \mathcal{D}(R, q, p)$ one can express by  
an MSO-sentence $\varphi_{d}$ the property that $d$ is the $p$-deletion $q$-folio of $(G, R)$.  

By Seese's theorem~\cite{S91}, there is an algorithm which given 
an MSO-sentence $\varphi$ and an integer $t$ decides whether there exists a graph of treewidth at most
$t$ satisfying $\varphi$. It is also known that there is a computable function $h(\varphi, t)$ such that 
if an MSO-sentence $\varphi$ 
is realizable by (or more accurately, admits a model consisting of) a graph of treewidth 
at most $t$, then there is such a graph with at most $h(\varphi, t)$ vertices. 
This can be derived using (among others) a theorem of Thatcher and Wright~\cite{TW68}, as
explained in the proof of Lemma 2.2 in~\cite{GroheFPT}.

Now clearly the set $\mathcal{D}(R, q, p)$ is computable. 
It follows that,  for each 
$d\in \mathcal{D}(R, q, p)$, we can use Seese's algorithm to 
compute the subset $\mathcal{D}_{z} \subseteq \mathcal{D}(R, q, p)$ of vectors 
$d \in \mathcal{D}(R, q, p)$ such that there exists a rooted graph $(G, R)$ of treewidth 
at most $z$ that satisfies $\varphi_{d}$. 
Moreover, for each $d\in \mathcal{D}_{z}$, 
we can compute a rooted graph $(G', R)$ whose $p$-deletion $q$-folio equals $d$ and such that  
$|G'| \leq h(\varphi_{d}, z)$. 
The lemma follows by setting $g(p, q, r, z) := \max\{ h(\varphi_{d}, z): d \in \mathcal{D}_{z}\}$.  
\end{proof}

A {\DEF separation} of a graph $G$ is a pair $(G_{1}, G_{2})$
of two induced subgraphs of $G$ such that $G = G_{1} \cup G_{2}$, its {\DEF order}
is $|V(G_{1}) \cap V(G_{2})|$. 
Let us point out that $G_{1} \subseteq G_{2}$ or $G_{2} \subseteq G_{1}$ 
is allowed in this definition, that is,
we do not require $V(G_{1}) - V(G_{2})$ and $V(G_{2}) - V(G_{1})$ to be nonempty. 

\begin{lemma}
\label{lem:reduction}
For every tree $T$ there exists a computable function $b$ such that, for every graph $G$
having a separation $(G_{1}, G_{2})$ of order $t$ with $|G_{1}| \geq b(t)$ such that   
$G_{1}$ has no $T$-minor,  
there exists a graph $G'$ satisfying  $\nu_{T}(G') = \nu_{T}(G)$, $\tau_{T}(G') = \tau_{T}(G)$, 
and $|G'| < |G|$. 
\end{lemma}
\begin{proof}
Since $G_{1}$ has no $T$-minor, its pathwidth is at most $|T|-2$ by Theorem~\ref{thm:quickly_forest}, 
and hence so is its treewidth.  
We will prove the lemma with $b(t):=g(t,q,t, |T| -2) + 1$, 
where $q:=t(|T|+1)$ and $g$ is the function from 
Lemma~\ref{lem:MSO}. 
Let $X := V(G_1) \cap V(G_2)$ and let $R$ be an arbitrary ordering of $X$. 
Using Lemma~\ref{lem:MSO},  let $(G'_1, R)$ be a rooted graph with the same 
$t$-deletion $q$-folio as $(G_{1}, R)$,  
and with $|G'_1| \leq g(t,q,t, |T| -2) < b(t) \leq |G_1|$. 
Assume without loss of generality that $V(G'_{1}) \cap V(G_{2}) = X$.  
Let $G'$ be the graph obtained from $G'_{1} \cup G_{2}$ by removing every 
edge linking two vertices in $X$ that is included in $G_{2}$ but not in $G'_{1}$. 
(Thus $G[X]$ and $G'[X]$ could possibly be distinct graphs.) 
Then $|G'| < |G|$, and it remains to show that  
$\nu_{T}(G') = \nu_{T}(G)$ and $\tau_{T}(G') = \tau_{T}(G)$.

First we prove that $\nu_{T}(G') = \nu_{T}(G)$ by showing that 
$\nu_{T}(L) \geq \nu_{T}(K)$ for $(K, L) \in \{(G, G'), (G', G)\}$. 
If $K = G$, let $K_{1} := G_{1}$ and $L_{1}:= G'_{1}$, otherwise let  
$K_{1} := G'_{1}$ and $L_{1}:= G_{1}$. 
Consider $k:=\nu_{T}(K)$ vertex-disjoint $T$-models $M_1, \dots, M_k$ in $K$, 
where $M_i = \{V^i_u: u\in V(T)\}$
for each $i \in [1,k]$. 
We may assume that these models are ordered so that 
$V(M_1), \dots, V(M_\l)$ each intersects $V(K_1)$, while none of $V(M_{\l+1}), \dots, V(M_k)$ intersects $V(K_1)$, 
for some index $\l \in [1,k]$. 
Observe that, since $T$ is connected and $K_1$ has no $T$-minor, each of $V(M_1), \dots, V(M_\l)$ must intersect $X$.
It follows that $\l \leq t$. 

For each $i \in [1,\l]$, let $X_i := V(M_i) \cap X$ (thus $X_i$ and $X_j$ are disjoint for $i\neq j$), 
let $A_i$ be the set of vertices $u\in V(T)$ such that $V^i_u$ intersects $X_i$, and, for every $u \in A_i$, let
$s_{u,i} := |V^i_u \cap X_i|$ and  arbitrarily partition $V^i_u \cap V(K_1)$ into 
$s_{u,i}$ parts $W^i_{u,1}, \dots, W^i_{u,s_{u,i}}$, so that each part induces a connected subgraph of $K_1$ and
contains exactly one vertex from $V^i_u \cap X_i$. Also let $B_i$
be the set of vertices $u\in V(T)$ such that $V^i_u$ intersects $V(K_1)$ but not $X_i$ (thus $A_i$ and $B_i$ are disjoint). 
For $u\in B_i$, let $s_{u,i} :=1$ and let $W^i_{u,1} := V^i_u$. 

Let $H^*$ be the graph with one vertex per triple $(i, u, j)$ where $i\in [1, \l], u\in  A_i \cup B_i$, 
and $j\in [1, s_{u,i}]$,
and where two distinct vertices $(i, u, j)$ and  $(i', u', j')$ are adjacent if and only if 
there is an edge in $K_1$ connecting
a vertex of $W^i_{u,j}$ to a vertex of $W^{i'}_{u',j'}$. 
Observe that $V^{i}_{u}$ corresponds to one vertex of $H^{*}$ if 
$V^{i}_{u} \subseteq V(K_{1}) - X$, and to $s_{u}$ vertices of $H^{*}$ otherwise. 
Let $R^*$ be the ordering of $X_1 \cup \cdots \cup X_{\l}$
induced by $R$.  
By the definition of $H^*$, the rooted graph $(H^*, R^*)$ is a minor of $(K_1, R^*)$. 
Since 
$$
|H^*| = \sum_{i=1}^{\l} \sum_{u \in A_i}  s_{u,i} + \sum_{i=1}^{\l} |B_i| 
= \sum_{i=1}^{\l} |X_i| + \sum_{i=1}^{\l} |B_i|
\leq |X| + \l |T|   \leq t + t |T| = q,
$$
$(H^*, R^*)$ is in the  $q$-folio of $(K_1, R)$. (More precisely, the $R^*$-entry of the
$q$-folio of $(K_1, R)$ contains $(H^*, R^*)$.)   
Hence $(H^*, R^*)$ is also 
in the  $q$-folio of $(L_1, R)$, since the $q$-folios of  $(K_1, R)$ and  
$(L_1, R)$ are the same (up to isomorphism, as always).  
Now consider a model of $(H^*, R^*)$ in $(L_1, R^*)$, and let  $Y^i_{u,j}$ denote the branch set 
corresponding to the vertex $(i,u,j)$ of $H^*$. 
Define, for each $i \in [1,k]$ and $u\in V(T)$, the set $Z^i_u$ as follows:
$$
Z^i_u := \left\{
\begin{array}{ll} 
(V^i_u - V(K_1)) \cup Y^i_{u,1} \cup \cdots \cup Y^i_{u,s_{u,i}} & \quad \textrm{if } i \leq \l \textrm{ and } u\in A_i \cup B_i, \\
V^i_u & \quad \textrm{otherwise}.
\end{array}
\right.
$$
For each $i \in [1,k]$, we have that $L[Z^i_u]$ is connected for 
each $u\in V(T)$, which follows from the fact that $K[V^{i}_{u}]$ is connected. 
Also, there exists an edge between $Z^i_u$ and $Z^i_v$ in $L$ for each $uv\in E(T)$.  
Let $M'_i := \{ Z^i_u: u \in V(T)\}$. 
Since  $Z^i_u$ is disjoint from  $Z^j_v$ for distinct $i,j \in [1,k]$ and $u,v\in V(T)$, we deduce
that $M'_1, \dots, M'_k$ are $k$ vertex-disjoint models of $T$ in $L$. 
Hence $\nu_{T}(K) \leq \nu_{T}(L)$, as desired.

Now we show that $\tau_{T}(G') = \tau_{T}(G)$ 
by showing that 
$\tau_{T}(L) \geq \tau_{T}(K)$ for $(K, L) \in \{(G, G'), (G', G)\}$. 
As before, if $K = G$, let $K_{1} := G_{1}$ and $L_{1}:= G'_{1}$, otherwise let  
$K_{1} := G'_{1}$ and $L_{1}:= G_{1}$.

Let $J$ be a minimum-size transversal of $L$, and let $J_1 := J \cap V(L_1)$. 
We have $|J_1| \leq |X| = t$, because $(J - J_1) \cup X$ is also a transversal of $L$, and thus must have size at least that of $J$. 
(Here we use the fact that $L_1$ has no $T$-minor, and that $T$ is connected.) 
Since the $t$-deletion $q$-folios of  $(K_1, R)$ and  $(L_1, R)$ are the same, there exists
a subset $I_1$ of vertices of $K_1$ with $|I_1| = |J_1|$ 
such that the $q$-folios of  $(L_1 - J_1, R - J_1)$ and  $(K_1 - I_1, R - I_1)$ are the same (by definition). 

We claim that the set $I := (J - J_1) \cup I_1$, which has the same size as $J$, 
is a transversal of $K$. 
Arguing by contradiction, assume not, and let $M:=\{ V_u : u\in V(T)\}$ denote a $T$-model in $K - I$. 
Let $A$ be the subset of vertices $u\in V(T)$ such that $V_u$ intersects $X$, and let $B$ be the set of those such that $V_u \subseteq V(K_1) - X$ (thus $A$ and $B$ are disjoint). 
For every $u\in A$, let $s_u := |V_u \cap X|$, and partition $V_u \cap V(K_1)$ into 
$s_{u}$ parts $W_{u,1}, \dots, W_{u,s_{u}}$, so that each part induces a connected subgraph of $K_1$ and
contains exactly one vertex from $V_u \cap X$.
Also let $s_u :=1$ and $W_{u,1} := V_u$ for every $u\in B$.
We define $H^*$ as the minor of $K_1 - I_1$ modeled by the $W_{u,j}$s, exactly as previously: 
$H^*$ has one vertex for every couple $(u, j)$ where $u\in A \cup B$ and $j\in \{1, \dots, s_u\}$, and
two distinct vertices $(u,j)$ and $(u', j')$ of $H^*$ are adjacent if and only if there is an edge between $W_{u,j}$
and $W_{u',j'}$ in $K_1 - I_1$. 

Let $R^*$ be the ordered subset of $R$ induced by the vertices in $V(M) \cap X$. 
Thus $(H^*, R^*)$ is a minor of $(K_1 - I_1, R^{*})$. 
Since $|H^*| = \sum_{u\in A} s_u + |B| \leq |X| + |T| \leq q$, the rooted graph $(H^*, R^*)$
is in the $q$-folio of $(K_1 - I_1, R - I_1)$, and hence also in 
the $q$-folio of $(L_1 - J_1, R - J_1)$. 
Consider a model of $(H^*, R^*)$ in $(L_1 - J_1, R^*)$; let 
$W'_{u,j}$ denote the branch set corresponding to the vertex $(u,j)$ of $H^{*}$ in that model. 
For every $u \in V(T)$, let
$$
V'_u := \left\{
\begin{array}{ll} 
(V_u - V(K_1)) \cup W'_{u,1} \cup \cdots \cup W'_{u,s_{u}} & \quad \textrm{if } u\in A_i \cup B_i, \\
V_u & \quad \textrm{otherwise}.
\end{array}
\right.
$$
It can be checked that $V'_u$ is disjoint from  $V'_v$ for every distinct $u,v\in V(T)$, that
$L[V'_u]$ is connected for 
every $u\in V(T)$, and that there is an edge between $V'_u$ and $V'_v$ in $L$ for every $uv\in E(T)$.  
We deduce that $M' := \{ V'_u: u \in V(T)\}$ is a model of $T$ in $L - J$, a contradiction. 
Therefore, the set $I$ is a transversal of $K$, and $\tau_{T}(K) \leq \tau_{T}(L)$, as claimed. 
\end{proof}

Even though we will not use this fact, we would like to point out that 
Lemma~\ref{lem:reduction} remains true more generally
if the tree $T$ is replaced by a connected planar graph $H$: 

\begin{lemma}
For every connected planar graph $H$ 
there exists a computable function $b$ such that, for every graph $G$
having a separation $(G_{1}, G_{2})$ of order $t$ with $|G_{1}| \geq b(t)$ such that   
$G_{1}$ has no $H$-minor,  
there exists a graph $G'$ satisfying  $\nu_{H}(G') = \nu_{H}(G)$, $\tau_{H}(G') = \tau_{H}(G)$, 
and $|G'| < |G|$. 
\end{lemma}

This can be shown using the exact same proof as for Lemma~\ref{lem:reduction}, the only difference
being that $G_{1}$ now has treewidth at most $c_{H}$ for some constant $c_{H}$ (by~\cite{RS_five}), 
and thus the function $b(t)$ should now be defined as $b(t):=g(t,q,t, c_{H}) + 1$.

\section{Proof of Main Theorem}
\label{sec:proof}

Let us start by recalling the overview of the proof of Theorem~\ref{thm:main}  
given in the introduction for the special case where $\F$ consists of a single tree $T$. 
The proof is by induction on $|G|$; 
we show that one can always either apply the reduction operation described 
in Lemma~\ref{lem:reduction} (in which case we are done by induction), or find a $T$-model 
in $G$ of size at most $c$. In the latter case, we delete all vertices of the model and 
apply induction on the remaining graph. 
The key step is thus proving that $T$-models of constant size can be found in $G$ when
the reduction operation cannot be applied. This is achieved by Lemma~\ref{lem:main_pw}, which is 
the main lemma of this section. 

Our basic strategy for finding a $T$-model of constant size when the reduction operation 
cannot be applied to  $G$ is to consider a ball $B$ of large but constant radius around an arbitrary vertex of $G$. 
If all vertices in $B$ have bounded degree then $B$ has bounded size, and
it turns out it is not difficult to show that $B$ must contain $T$ as a minor 
(in which case we found a $T$-model of constant size) because otherwise the reduction operation could have been applied to $G$. However, difficulties arise when $B$ contains vertices 
of high degrees. This is why in Lemma~\ref{lem:main_pw} we do not directly try and find 
a $T$-model of constant size, but instead look for a subgraph of $G$ of constant size having 
some specified pathwidth $t$. This is enough, since by Theorem~\ref{thm:quickly_forest} 
every graph with pathwidth at least $|T|-1$ contains $T$ as a minor, and allows 
us to set up an inductive argument that handles vertices of high degree in $G$. 

First we need to introduce a few definitions and results. Lemmas~\ref{lem:marked}, 
\ref{lem:refinement}, 
\ref{lem:tree}, 
\ref{lem:pathwidth_rooted_minors}, and
\ref{lem:pw_increase} below 
are short lemmas that will be used in the proof of Lemma~\ref{lem:main_pw}. 
 
Given a graph $G$ and $t \in \N$, a {\DEF \tsep{t}} of $G$ is a separation $(G_1, G_2)$ of $G$ 
such that $G_1$ has a path decomposition $B_1, \dots, B_p$ of width at most $t$ with $V(G_1) \cap V(G_2)=B_1 \cup B_p$. 
Notice that the order of such a separation is at most $2t+2$. 

\begin{lemma} 
\label{lem:marked}
Let $G$ be a graph of pathwidth at most $t$ with $k \geq 0$ marked vertices. 
Then there is a {\tsep{t}} $(G_1, G_2)$ of $G$  with 
$|G_1| \geq \frac{|G| - k(t+1)}{k+1}$ such that no vertex in $V(G_1) - V(G_2)$ is marked. 
\end{lemma}
\begin{proof}
Consider a path decomposition  $B_1, \dots, B_{p}$ of $G$ with width at most $t$. 
If $k=0$, then  $G_{1} := G$ and $G_{2}:= G[B_{1} \cup B_{p}]$ defines the desired 
 {\tsep{t}} $(G_1, G_2)$ of $G$. Now assume $k \geq 1$. 

For each marked vertex $v$, choose an arbitrary bag that contains $v$ and mark it. 
Let $m_1, \dots, m_{\l}$ denote the indices of the marked bags. Observe that $1 \leq \l \leq k$
(note that $\l < k$ if some bag is marked more than once). Also let  $m_0 := 1$ and $m_{\l +1} := p$. 

For $1 \leq i < j \leq p$, let $Y_{i,j}:= (B_{i+1} \cup \cdots \cup B_{j-1}) - (B_i \cup B_j)$. 
Observe that $Y_{i,j}$ is disjoint from $Y_{i', j'}$ when $j \leq i'$  
(this is a consequence of the axioms of path decompositions). 
In particular, the $\l+1$ sets $Y_{1, m_1}, Y_{m_1, m_2}, \dots, Y_{m_{\l -1}, m_{\l}}, Y_{m_{\l}, p}$ 
are pairwise disjoint. 
Hence, among them there is a set $Y_{a,b}$ satisfying  
$$
|Y_{a,b}| \geq \frac{ \sum_{j=0}^{\l} |Y_{m_j, m_{j+1}}|}{\l +1} 
\geq \frac{|G| - \sum_{j=1}^{\l} |B_{m_{j}}|}{\l +1} 
\geq \frac{|G| - \l(t+1)}{\l +1}
\geq \frac{|G| - k(t+1)}{k +1}.   
$$
On the other hand, 
none of $Y_{1, m_1}, Y_{m_1, m_2}, \dots, Y_{m_{\l -1}, m_{\l}}, Y_{m_{\l}, p}$ contains a marked vertex
(as follows again from the axioms of path decompositions).  
Thus  $G_1 := G[Y_{a,b} \cup B_a \cup B_b]$ and $G_2 := G[V(G) -  Y_{a,b}]$ define
a separation $(G_1, G_2)$ of $G$ with $V(G_1) \cap V(G_2) = B_a \cup B_b$ such that $|G_1| \geq \frac{|G| - k(t+1)}{k +1}$
and $V(G_1) - V(G_2)$ ($= Y_{a,b}$) has no marked vertex.   
Since $B_a, B_{a+1}, \dots, B_b$ is a path decomposition of $G_1$ with width at most $t$, we 
deduce that $(G_1, G_2)$ is the desired {\tsep{t}} of $G$.
\end{proof}

\begin{lemma}
\label{lem:refinement}
Let $(G_1, G_2)$ be a {\tsep{t}} of a graph $G$
with $k \geq 0$ marked vertices, and
let $\l$ be an integer satisfying $1 \leq \l \leq \lceil |G_{1}|/(k+1) \rceil$. 
Then there exists a {\tsep{t}} $(G'_1, G'_2)$ of $G$  with
$G'_1 \subseteq G_{1}$, $G_2 \subseteq G'_{2}$,  and $|G'_{1}| = \l$,  
such that no vertex in $V(G'_1) - V(G'_2)$ is marked. 
\end{lemma}
\begin{proof} 
{\bf Case~1: $\l \leq t+1$}. Let $Y$ be an arbitrary subset of $\l$ vertices of $G_{1}$,
let $G'_{1} := G[Y]$ and let $G'_{2}:=G$. Then $V(G'_1) - V(G'_2)$ is empty
and $G'_{1}$ has a trivial path decomposition of width at most $t$ 
consisting of a single bag $B_{1}:=Y$. Hence $(G'_1, G'_2)$ is a 
{\tsep{t}} of $G$ as desired. 

{\bf Case~2: $\l > t+1$}. 
Let $B_{1}, \dots, B_{p}$ be a path decomposition of $G_{1}$ of width at most $t$
such that $B_{1} \cup B_{p} = V(G_{1}) \cap V(G_{2})$. 
We may assume that $|B_{i} \bigtriangleup B_{i+1}| = 1$ for each $i \in [1, p-1]$,
by modifying the path decomposition if necessary 
(observe that this can be done without changing the first and last bags).  

Mark bags $B_{1}$, $B_{p}$, and for each marked vertex $v$ mark an arbitrarily chosen 
bag that contains $v$ (thus the same bag can be chosen several times).  
Let $q$ be the total number of marked bags, and let $m_{1}, \dots, m_{q}$ denote the
indices of these bags, in order. (Thus $m_{1}=1$ and $m_{q}=p$.)
Since $|B_{1}| \leq t +1$ and $\l > t+1$, it follows that $p \geq 2$, and hence $q \geq 2$.
Let $Y_{a,b} := B_{a} \cup B_{a+1} \cup \cdots \cup B_{b}$
for $1 \leq a < b \leq p$. 
Since $Y_{m_{1},m_{2}} \cup \cdots \cup Y_{m_{q-1},m_{q}} = V(G_{1})$ and $q-1 \leq k+1$, there exists 
$j \in [1, q-1]$ such that 
$$
|Y_{m_{j},m_{j+1}}| \geq \left\lceil \frac{|G_{1}|}{q-1} \right\rceil 
\geq \left\lceil \frac{|G_{1}|}{k+1} \right\rceil \geq \l.
$$

Now, since $|B_{i} \bigtriangleup B_{i+1}| = 1$ for each $i \in [1,  p-1]$
and since $\l > |B_{m_{j}}|$, there exists an index $b\in [m_{j} + 1, m_{j+1}]$ 
such that $|Y_{m_{j},b}| = \l$. 
It follows from the axioms of a path decomposition that 
no vertex in $Y_{m_{j},b} - (B_{m_{j}} \cup B_{b})$ is marked.
Therefore, setting $G'_{1} := G[Y_{m_{j},b}]$ and 
$G'_{2} := G - (Y_{m_{j},b} - (B_{m_{j}} \cup B_{b}))$, 
and using that $B_{m_{j}}, \dots, B_{b}$ is a path decomposition of $G'_{1}$ of width at most $t$,
we deduce that $(G'_1, G'_2)$ is a {\tsep{t}} of $G$ with the required properties. 
\end{proof}

With a slight abuse of notation, we simply write $(G, v)$ to denote the rooted graph $(G, R)$
where $R$ consists only of the vertex $v$. 
The {\DEF height} of a rooted tree $(T, v)$  is the maximum length of a path from $v$ to a leaf, where the length
of a path is the number of its edges.  (Thus the height of $(T, v)$ is zero if $|T|=1$.) 
For $u \in V(T)$, the notions of {\DEF parent}, {\DEF ancestors}, 
{\DEF children}, {\DEF descendants} of $u$ in $(T, v)$
are defined as expected. When the root $v$ is clear from the context, we simply denote by $T_u$ 
the subtree of $(T, v)$ induced by $u$ and all its descendants. 
For $h \geq 1$, the {\DEF complete binary tree of height $h$}, denoted $\mc B_h$, is the unique rooted
tree of height $h$ where the root has degree $2$, every other non-leaf vertex has degree $3$, and
every path from the root to a leaf has length exactly $h$.
It is known, and not difficult to prove, that the pathwidth of $\mc B_h$ equals $h$. 

\begin{lemma}
\label{lem:tree}
Let $(T, v)$ be a rooted tree of height $h$ and maximum degree $\Delta$. 
If $(T, v)$ does not contain $\mc B_{k+1}$ as a (rooted) minor, then $|T| \leq (h+1)^{k+1}(\Delta+1)^{k+1}$. 
\end{lemma}
\begin{proof}
The proof is by induction on $k$.
If $k=0$, then the tree $T$ is a path, and thus satisfies $|T| = h +1 \leq (h+1) (\Delta + 1)$. 
For the inductive step, assume $k \geq 1$. 
We may assume that  $\mc B_{k}$ is a minor of $(T, v)$, because otherwise we are done by induction. 
Let $u \in V(T)$ be a vertex such that $(T_u, u)$ contains $\mc B_{k}$ as a minor and $u$ is at maximum distance from $v$ in $T$.
Let $q$ denote that distance. Let $u_1, \dots, u_{\l}$ denote the children of $u$. Observe that $\l \geq 2$ since 
$(T_u, u)$ has a $\mc B_{k}$-minor, $k \geq 1$, and $q$ is maximum. Let $v_1, \dots, v_{q+1}$ be the vertices on the path from $v$ to $u$
in $T$, in order. Thus $v_1=v$ and $v_{q+1}=u$. 
For each $i \in [1, q]$, let $w_{i,1}, \dots, w_{i,a_i}$ denote the children of $v_i$ in $T$ that are distinct from $v_{i+1}$
(note that possibly $a_i = 0$). 

We claim that, for each $i \in [1,q]$ and $j \in [1, a_i]$, the rooted tree $(T_{w_{i,j}}, w_{i,j})$ has no $\mc B_k$ minor.
Indeed, otherwise a model of $\mc B_k$ in $(T_{w_{i,j}}, w_{i,j})$ could be combined with a model of
$\mc B_k$ in $(T_{v_{i+1}}, v_{i+1})$ (which exists, by the definition of $u$) and the $v_1$--$v_i$ path in $T$ to give a 
model of $\mc B_{k+1}$ in $(T, v)$, a contradiction. By induction, we thus have $|T_{w_{i,j}}| \leq h^{k}(\Delta+1)^{k}$, since 
the height of $(T_{w_{i,j}}, w_{i,j})$ is at most $h-1$. 

By the definition of $u$, the rooted tree $(T_{u_i}, u_i)$ has no $\mc B_k$ minor either for each $i \in [1,  \l]$. (Recall that $u_1, \dots, u_{\l}$ are the children of $u$.)
Hence
$|T_{u_i}| \leq h^{k}(\Delta+1)^{k}$ by induction. It follows
\begin{align*}
|T| &= q + 1 + \sum_{i=1}^q \sum_{j=1}^{a_i}  |T_{w_{i,j}}| + \sum_{i=1}^{\l}|T_{u_i}| \\
&\leq h +1 + \left[
\sum_{i=1}^q \sum_{j=1}^{a_i} h^{k}(\Delta+1)^{k}\right] + \l h^{k}(\Delta+1)^{k} \\
&\leq h^{k}(\Delta+1)^{k} + h (\Delta + 1) h^{k}(\Delta+1)^{k} + \Delta h^{k}(\Delta+1)^{k} \\
&= h^{k+1}(\Delta+1)^{k+1} +h^{k}(\Delta+1)^{k+1} \\
&\leq (h+1)^{k+1}(\Delta+1)^{k+1},
\end{align*}
as desired. (In the second inequality, we used that $h +1 \leq h^{k}(\Delta+1)^{k}$, which follows from the fact
that $h,k \geq 1$ and $\Delta \geq 2$.)
\end{proof}

\begin{lemma}
\label{lem:pathwidth_rooted_minors}
If $(T,v)$ is a rooted tree on $t$ vertices and
$(G,w)$ is a connected rooted graph with pathwidth at least $2t-2$,
then $(G,w)$ contains $(T, v)$ as a minor.
\end{lemma}
\begin{proof}
Let $(T',v')$ be a copy of $(T,v)$. 
Let $T^*$ be the (unrooted) tree obtained from the disjoint union of $T$ and $T'$
by identifying $v$ and $v'$.
Thus $T^*$ has $2t-1$ vertices.  Since $G$ has pathwidth at least $2t-2$, it 
contains a $T^*$-minor (by Theorem~\ref{thm:quickly_forest}), 
and hence has a model of $T^*$. Since $G$ is connected, we may
assume that every vertex of $G$ is in a branch set of this model. 
Let $B_v$ denote the branch set of vertex $v=v'$. 
Exchanging $T$ and $T'$ if necessary, we may assume that the root 
$w$ of $(G, w)$ is in a branch set corresponding to a vertex from $T$. 
Thus we can find a path $P$ in $G$  with one endpoint being $w$ and the other endpoint in $B_v$ such that 
$P$ avoids all branch sets of vertices in $V(T') - \{v'\}$. 
Now, replacing the branch set $B_{v}$ by $B_{v}  \cup V(P)$ and taking all branch sets
corresponding to vertices of $T'$ distinct from $v'$, we obtain a model of $(T', v')$ in $(G, w)$. 
Therefore, $(G,w)$ contains $(T,v)$ as a minor. 
\end{proof}

The following lemma is well known; a proof is included for completeness
(see Corollary 3.1 in~\cite{EST94} for a similar result).  

\begin{lemma}
\label{lem:pw_increase}
Let $G_1, G_2, G_3$ be three connected graphs, each with pathwidth at least $k$,
and let $v_i$ be an arbitrary vertex of $G_i$, for $i=1,2,3$. 
Let $G$ be the graph obtained from the disjoint union of $G_1, G_2, G_3$ by adding a new vertex $v$ adjacent to $v_1, v_2, v_3$. Then $G$ has pathwidth at least $k+1$. 
\end{lemma}
\begin{proof}
Arguing by contradiction, assume $G$ has a path decomposition $B_1, \dots, B_p$ of width at most $k$. 
For every $w \in V(G)$, let $[\l_w, r_w]$ denote the interval of indices $j$ such that $w \in B_j$. 
For each $i \in [1,3]$, let $I_i := \cup \{ [\l_w, r_w] : w \in V(G_i)\}$. Since $G_i$ is connected, $I_i$ is again an interval;
we denote by $\l_i$ and $r_i$ its left and right endpoints, respectively. 
Reindexing $G_1, G_2, G_3$ if necessary,  
we may assume  that $\l_1\leq \l_2$ and $r_2 \leq r_3$. 
Since $vv_i \in E(G)$, it follows that $I_i \cap [\l_v, r_v] \neq \varnothing$ for each $i \in [1,3]$. 
This implies that $I_1 \cup I_3 \cup [\l_v, r_v]$ is again an interval, and this interval contains $I_2$. 
Hence, every bag $B_j$ ($j \in [1, p]$) that includes a vertex of $V(G_2)$ contains also at least
one vertex from $V(G) - V(G_2)$. Therefore, 
$B_1 \cap V(G_2), \dots, B_p\cap V(G_2)$ is a path decomposition of $G_2$ of width at most $k-1$, a contradiction. 
\end{proof}

The next lemma is our main tool. Informally, it shows that 
every connected graph $G$ with pathwidth at least $t$ and no 
{\tsep{(t-1)}} $(G_1, G_2)$ of $G$ with $|G_1|$ ``big'' has a connected subgraph $G'$ 
of constant size with pathwidth at least $t$. It turns out that a slightly stronger statement 
is easier to prove by induction, namely that for every vertex $w$ of $G$ one can find 
such a subgraph $G'$ containing it.   
We note that the vertex $w$ will be an arbitrary vertex when this lemma is used 
in the proof of Theorem~\ref{thm:main}. 

\begin{lemma}
\label{lem:main_pw}
There exists a computable function $f:\N \times \N \to \N $ such that for
every $t, r\in \N$, 
every connected graph $G$ of pathwidth at least $t$, and every vertex $w\in V(G)$, 
at least one of the following holds:
\begin{enumerate}[(a)]
\item there exists a connected subgraph $G'$ of $G$ 
with $w\in V(G')$ such that $G'$ has pathwidth at least $t$ and $|G'| \leq f(t, r)$, 
\item there exists a {\tsep{(t-1)}} $(G_1, G_2)$ of $G$ with $|G_1| \geq r$. 
\end{enumerate}
\end{lemma}
\begin{proof}
The proof is by induction on $t$. For $t=0$, the vertex $w$ itself provides a connected subgraph $G'$ of $G$ of pathwidth $0$, 
and thus the claim holds with $f(0, r) :=1$. Now assume $t \geq 1$. 
If $r \leq 1$, letting $G_{1}$ be the graph induced by an arbitrary vertex of $G$ and
letting $G_{2}:=G$, we obtain a {\tsep{(t-1)}} $(G_1, G_2)$ of $G$
with $|G_1| \geq r$.  Thus we may also assume $r \geq 2$. 

Let
\begin{align*}
r_{1} &:= (r+t-1) + (r+t)2t + r \\
r_{2} &:= (r+t-1)\big(1 + f(t-1, r_{1})\big) + (r+t)2t + r \\
r_{3} &:= (r+t-1)\big(1 + f(t-1, r_{1}) + f(t-1, r_{2})\big) + (r+t)2t + r.
\end{align*}    
Let $\Delta, \eps$, and $d$ be constants defined as follows:
\begin{align*}
\Delta &:= (r+t)\left(\sum_{i=1}^{3}f(t-1, r_{i}) + 2t + 1 \right) + r - 1; \\
\eps &:= \frac{1}{2r + t}; \\ 
d &:= \max\left\{ 
\frac{2t\ln(\Delta+1)}{\ln(1 + \eps)},
\left(\frac{2t}{\ln(1 + \eps)}\right)^2
\right\}.
\end{align*}
We will prove the claim with  
$$
f(t,r) := \left\lceil \max\left\{ \Delta^{d+1}, \Delta + d + 1 \right\} \right\rceil. 
$$
(We remark that the ceiling is there only to ensure that $f(t,r)$ is an integer.) 
We may assume that $|G| > f(t, r)$, since otherwise we are done with $G'=G$. 

{\bf Case~1: Every vertex at distance at most $d$ from $w$ has degree at most $\Delta$.} 
For $i=0, 1, \dots, d$, let $H_i$ be the subgraph of $G$ induced by all vertices  
at distance at most $i$ from $w$, and let $J_i$ be an arbitrary breadth-first
search tree of $H_i$ from $w$. Since $|H_d| \leq \Delta^{d+1} \leq  f(t, r) < |G|$,  
we deduce that $V(H_i) - V(H_{i-1})$ is not empty for each $i \in [1, d]$.  
The graph $H_d$ is connected, includes the vertex $w$, and has at most $f(t, r)$ vertices. 
Thus we are done if the pathwidth of $H_{d}$ is at least $t$. 
So let us assume that $H_{d}$ has pathwidth at most $t-1$ (and thus, in particular, 
$J_{d}$ has pathwidth at most $t - 1$). 

The tree  $\mc B_{t}$, which has pathwidth $t$, cannot be a minor of $(J_d, w)$. 
Hence we have
\begin{equation}
\label{eq:Hd}
|H_d|=|J_d| \leq (d+1)^{t}(\Delta+1)^{t}
\end{equation}
by Lemma~\ref{lem:tree}. 

By the definition of $d$, 
$$
d \geq \frac{d}{2} + \frac{\sqrt d \ln(d+1)}{2} 
\geq \frac{t\ln(\Delta+1)}{\ln(1 + \eps)} + \frac{t\ln(d+1)}{\ln(1+\eps)}, 
$$
which implies $(1+\eps)^d \geq (d+1)^{t}(\Delta+1)^{t}$.
Thus, if $|H_i| > (1+\eps)|H_{i-1}|$ for each $i\in [1, d]$, then 
$$
|H_d| > (1+\eps)^d \geq (d+1)^{t}(\Delta+1)^{t}, 
$$
contradicting~\eqref{eq:Hd}. 
Hence there exists $j\in [1, d]$ such that $|H_j| \leq (1+\eps)|H_{j-1}|$. 
Mark all $k:= |H_j| - |H_{j-1}|$ vertices of $H_j$ that are in 
$V(H_j) - V(H_{j-1})$. We have $k \geq 1$, since $V(H_j) - V(H_{j-1})$ is not empty, and 
also $k \leq \eps |H_{j-1}| \leq \eps |H_{j}|$.  
Since the pathwidth of $H_{j}$ is at most that of $H_{d}$,  
by Lemma~\ref{lem:marked} there is a {\tsep{(t-1)}} $(L_1, L_2)$ of $H_j$  
such that no vertex in $V(L_1) - V(L_2)$ is marked and such that
$$
|L_1| \geq \frac{|H_j| - kt}{k+1} 
\geq \frac{|H_j| - \eps t|H_j|}{2k}
\geq \frac{1 - \eps t}{2\eps}
= r.
$$
Now, the fact that no vertex in $V(L_1) - V(L_2)$
is marked implies that $V(L_1) - V(L_2) \subseteq V(H_{j-1})$, and hence 
no vertex in $V(L_1) - V(L_2)$ is adjacent in $G$ to a vertex in $V(G) - V(L_1)$. 
Therefore, 
$G_1:=L_1$ and $G_2:= G - (V(L_1) - V(L_2))$ defines a {\tsep{(t-1)}} $(G_1, G_2)$ of $G$ 
with $|G_1| \geq r$. This concludes the proof of Case~1. 

{\bf Case~2: Some vertex at distance at most $d$ from $w$ has degree more than $\Delta$.} 
Let $x$ be such a vertex and let $x_1, \dots, x_p$ denote its neighbors, where $p \geq \Delta + 1$.  

First we prove a few easy claims. 

\begin{claim}
\label{claim:link_up}
If there exists a connected subgraph $H$ of $G$ of pathwidth at least $t$ 
such that $|H| \leq f(t, r) -d$  and  $x \in V(H)$, 
then the lemma holds.  
\end{claim}
\begin{proof}
Let $P$ be a shortest $w$--$x$ path in $G$ and let $G' := G[ V(H) \cup V(P) ]$. Then $G'$ is connected and has
pathwidth at least $t$; moreover, $w\in V(G')$ and $|G'| \leq |H| + d \leq f(t, r)$.
\end{proof}

\begin{claim}
\label{claim:exists1} 
If there exists $X\subseteq V(G)$ with $|X| \leq \frac{f(t, r) - d - r}{r+t} - 2t$ 
such that $G[X]$ is connected, 
$x\in X$, 
and
$H:=G - X$ has a {\tsep{(t-1)}} $(H_1, H_2)$ with 
$|H_1| \geq (r + t - 1) |X| + (r+t)2t + r$, 
then the lemma holds. 
\end{claim}
\begin{proof}
Using Lemma~\ref{lem:refinement} with $k=0$ and $\l= (r + t - 1) |X| + (r+t)2t + r$, 
we may assume that the {\tsep{(t-1)}} $(H_1, H_2)$  of $H$
has been chosen so that $|H_1| = (r + t - 1) |X| + (r+t)2t + r$.  

Let $J:= G[ V(H_1) \cup X ]$. If $J$ has pathwidth at least $t$,  then 
since $H_{1}$ has pathwidth at most $t-1$ there is a unique component $J'$ of $J$
of pathwidth at least $t$, namely the one containing $X$ (recall that $G[X]$ is connected).
Since $x\in V(J')$ 
and $|J'| \leq |J| = |H_1| + |X| = (r + t) |X| + (r+t)2t + r  \leq f(t, r) - d$,   
we are done by Claim~\ref{claim:link_up}.  
Thus we may assume that $J$ has pathwidth at most $t-1$. 

Let $Y:=X \cup (V(H_1) \cap V(H_2))$ and mark all vertices of $J$ that are in $Y$. 
Since $|J| = (r + t) |X| + (r+t)2t + r  \geq (r + t) |Y| + r$, 
by Lemma~\ref{lem:marked} there is a {\tsep{(t-1)}} $(J_1, J_2)$ of $J$ with $|J_1| \geq r$
such that $Y \subseteq V(J_2)$. 
Observe that no vertex in $V(J_1) - Y$ is adjacent in $G$ to a vertex in $V(G) - V(J_1)$. Hence, 
 $G_1:=J_1$ and $G_2:= G - (V(J_1) - Y)$ defines a {\tsep{(t-1)}} $(G_1, G_2)$ 
of $G$ with $|G_1| \geq r$. 
\end{proof}

\begin{claim}
\label{claim:exists2}
If there exists $X\subseteq V(G)$ with $|X| \leq 1 + \sum_{i=1}^{3}f(t-1, r_{i})$ 
such that $G[X]$ is connected, $x\in X$, 
and
every component of $G - X$ that includes some vertex in  $\{x_1, \dots, x_p\}$ 
has pathwidth at most $t-1$, then the lemma holds. 
\end{claim}
\begin{proof}  
Let $q$ be the number of neighbors of $x$ that are not in $X$; let us assume
without loss of generality that these neighbors are $x_{1}, \dots, x_{q}$. 
Let $H$ be the union of all components of $G-X$ that include at least one of these vertices.
Thus $H$ has pathwidth at most $t-1$ by the assumption of the claim. 
Since $|H| \geq q \geq p - |X| \geq \Delta + 1 - |X| \geq (r + t - 1) |X| + (r+t)2t + r$, 
by Lemma~\ref{lem:marked} there is a {\tsep{(t-1)}} $(H_1, H_2)$ of $H$ 
with $|H_1| \geq (r + t - 1) |X| + (r+t)2t + r$. The claim follows then from Claim~\ref{claim:exists1},
since  $|X| \leq \sum_{i=1}^{3}f(t-1, r_{i}) + 1 \leq \frac{f(t, r) - d - r}{r+t} - 2t$.  
\end{proof}

Let $X_{0} := \{x\}$. 
Apply the following argument first with $j=1$, then $j=2$, then $j=3$. 
Since $|X_{j-1}| \leq 1 + \sum_{i=1}^{j-1}f(t-1, r_{i})$, 
the graph $G[X_{j-1}]$ is connected, and $x \in X_{j-1}$,
by Claim~\ref{claim:exists2} we may assume that there is a component $C$ of $G - X_{j-1}$
with pathwidth at least $t$ that includes a neighbor $x_{n_{j}}$ of $x$ (otherwise we are done). 
By induction (on $t$), either $C$ has a connected subgraph $H_{j}$ of 
pathwidth at least $t-1$ with  
$|H_{j}| \leq f(t-1, r_{j})$ and $x_{n_{j}} \in V(H_{j})$,
or there is a {\tsep{(t-2)}} $(C_{1}, C_{2})$ of $C$ with $|C_{1}| \geq r_{j}$. 
In the second case, $(C_{1}, C'_{2})$ is a {\tsep{(t-2)}} of $G - X_{j-1}$, where
$C'_{2}$ is the union of $C_{2}$ with all components of $G - X_{j-1}$ distinct from $C$. 
Since  $|X_{j-1}| \leq 1 + \sum_{i=1}^{j-1}f(t-1, r_{i}) \leq \frac{f(t, r) - d - r}{r+t} - 2t$ and
by the definition of $r_{j}$ we have
$|C_{1}| \geq r_{j} \geq (r+t-1)|X_{j-1}| + (r+t)2t + r$,  
we are done by Claim~\ref{claim:exists1}.  
In the first case, let $X_{j} := X_{j-1} \cup V(H_{j})$.
Observe that $|X_{j}| \leq 1 + \sum_{i=1}^{j}f(t-1, r_{i})$ 
and that $G[X_{j}]$ is connected (because of the vertex $x$). 

Now that the sets $X_{0}, X_{1}, X_{2}, X_{3}$ are defined, let $P$ be a shortest path in $G$ from $w$ to $X_{3}$. 
Thus $|P| \leq d + 1$, since $x \in X_{3}$ and $x$ is at distance at most $d$ from $w$ in $G$.  
Since $H_{i}$ is connected and has pathwidth at least $t-1$ for each $i \in [1,3]$,
and since $H_{1}, H_{2}, H_{3}$ are pairwise vertex-disjoint, 
the graph $G[X_{3}]$ has pathwidth at least $t$ by Lemma~\ref{lem:pw_increase}. 
It follows that $G' := G[X_{3} \cup V(P)]$ is a connected subgraph of $G$ that includes $w$, 
with pathwidth at least $t$, and satisfying 
$|G'| \leq |X_{3}| + |P| - 1 \leq 1 + \sum_{i=1}^{3}f(t-1, r_{i}) + d \leq f(t,r)$. 
This concludes the proof. 
\end{proof}

Now we turn to the proof of Theorem~\ref{thm:main}. 

\begin{proof}[Proof of Theorem~\ref{thm:main}]
The proof is in two steps: First we prove the theorem in the special case where 
$\F$ consists of a single tree, which we then use to handle the general case.

First consider the case where $\F=\{T\}$, where $T$ is a tree on $t$ vertices. 
Let $r:= \max \{ b(i): 0 \leq i \leq 2t - 2\}$, where $b$ is the function
from Lemma~\ref{lem:reduction}. 
Let $G$ be an arbitrary graph. 
Let $c := \max\{1,2(t-1), r, f(t-1, r)\}$, 
where $f$ is the function from Lemma~\ref{lem:main_pw}.   
We will show that $\tau_{\F}(G) \leq c \cdot \nu_{\F}(G)$ 
by induction on $|G|$.
If $|G| = 1$ then the claim obviously holds since $c \geq 1$.
Now assume that $|G| > 1$. Since $T$ is connected, we have that 
$\nu_{\F}(G) = \nu_{\F}(G_{1}) + \cdots + \nu_{\F}(G_{\l})$ and 
$\tau_{\F}(G) = \tau_{\F}(G_{1}) + \cdots + \tau_{\F}(G_{\l})$, where
$G_{1}, \dots, G_{\l}$ are the components of $G$.  
We may thus assume that $G$ is connected, since otherwise
we are done by applying induction on its components.

If the pathwidth of $G$ is at most $t-2$, then 
we directly obtain that $\tau_{\F}(G) \leq 2(t-1) \cdot \nu_{\F}(G) \leq c \cdot \nu_{\F}(G)$ 
by Lemma~\ref{lem:bounded_pathwidth}. 
Hence assume that $G$ has pathwidth at least $t-1$. 

Apply Lemma~\ref{lem:main_pw} on $G$ (with $w$ an arbitrary vertex of $G$) 
to obtain one of the two possible outcomes. If the outcome is 
a {\tsep{(t-2)}} $(G_{1}, G_{2})$ of $G$ with $|G_{1}| \geq r$ then
by Lemma~\ref{lem:refinement} (with no marked vertices) we may assume that $|G_{1}|=r \leq c$. 
If $G_{1}$ contains $T$ as a minor then 
$\nu_{\F}(G - V(G_{1})) \leq \nu_{\F}(G) - 1$. By induction, there is an $\F$-transversal $Y$
of $G- V(G_{1})$ of size at most $c\cdot \nu_{\F}(G - V(G_{1})) \leq c \cdot \nu_{\F}(G) - c$, and hence 
$V(G_{1}) \cup Y$ is an $\F$-transversal of $G$ of size at most $c \cdot \nu_{\F}(G)$.
If $G_{1}$ has no $T$-minor, then  
since $q := |V(G_{1}) \cap V(G_{2})| \leq 2t - 2$ and $r \geq b(q)$, by Lemma~\ref{lem:reduction} 
there is a graph $G'$ with $\nu_{\F}(G') = \tau_{\F}(G)$, 
$\tau_{\F}(G') = \tau_{\F}(G)$, and $|G'| < |G|$. Since by induction 
$\tau_{\F}(G') \leq c \cdot \nu_{\F}(G')$, we are done in this case. 

Suppose now that the outcome of Lemma~\ref{lem:main_pw} is 
a subgraph $G'$ of $G$ of pathwidth at least $t-1$ with $|G'| \leq f(t-1, r) \leq c$. 
Then, since $G'$ contains $T$ as a minor (by Theorem~\ref{thm:quickly_forest}), 
$\nu_{\F}(G - V(G')) \leq \nu_{\F}(G) - 1$. By induction, there is an $\F$-transversal $Y$
of $G-V(G')$ of size at most $c\cdot \nu_{\F}(G - V(G')) \leq c \cdot \nu_{\F}(G) - c$, and hence 
$V(G') \cup Y$ is an $\F$-transversal of $G$ of size at most $c \cdot \nu_{\F}(G)$.
This concludes the proof of the case where $\F$ consists of a single tree. 

Now let $\F$ be an arbitrary finite set of graphs containing a forest $F$. 
Let $q:=|\F|$ and let $r$ denote the maximum number of components of a graph in $\F$. 
Let $t:=|F|$. Let $T$ be an arbitrary tree obtained from $F$ by adding edges. 
By the proof above, there is a constant $c'$ such that 
$\tau_{T}(G) \leq c' \cdot \nu_{T}(G)$ for every graph $G$. 
Using Lemma~\ref{lem:F_to_tree}, we obtain
$$
\tau_{\F}(G) \leq \tau_{T}(G) + 2qrt\cdot \nu_{\F}(G) 
\leq c' \cdot \nu_{T}(G) + 2qrt\cdot \nu_{\F}(G) 
\leq (c' + 2qrt)\cdot \nu_{\F}(G) 
$$
for every graph $G$. 
Therefore, the theorem holds with $c:= c' + 2qrt$. 
\end{proof}

We remark that, while Corollary~\ref{cor:main} was 
deduced from Theorem~\ref{thm:main} and the computability of the obstruction set
for graphs of pathwidth at most $t$ (see~\cite{AGK08, L98}), it can alternatively be 
derived directly from Lemma~\ref{lem:main_pw}.

\section{Algorithmic Implications}
\label{sec:algorithms}

While the focus of this article is not algorithms,  
we would nevertheless like to point out a few algorithmic implications of our results.  

First, the proof of Theorem~\ref{thm:main} can be turned into a polynomial-time 
algorithm that, for fixed $\F$, computes 
in polynomial time an $\F$-packing and an $\F$-transversal 
of a given input graph $G$ differing in size by at most a factor $c$. 
This is explained in part by the fact that 
the graph $(G', R)$ in Lemma~\ref{lem:MSO} can be computed in polynomial
time using standard monadic second-order logic 
techniques (as done in~\cite{GroheFPT}, for instance), 
and the same is true for $G'$ in Lemma~\ref{lem:reduction}. 
Moreover, it is not difficult to extend Lemma~\ref{lem:reduction} to obtain that, given 
an $\F$-packing in $G'$, one can compute an $\F$-packing in $G$ of the same size in polynomial time, 
and similarly given an $\F$-transversal of $G'$, one can 
compute an $\F$-transversal in $G$ of no larger size in polynomial time. 
(This is needed when applying a reduction operation in the proof of Theorem~\ref{thm:main}, since 
after having obtained an $\F$-packing and an $\F$-transversal of a reduced graph $G'$, 
we have to ``lift back'' these to the input graph $G$.)

Furthermore, the outcomes of Lemmas~\ref{lem:marked} and~\ref{lem:refinement} 
can easily be computed in polynomial 
time (for fixed $t$). The same is true for 
Lemma~\ref{lem:main_pw} (for fixed $t$ and $r$), 
as its main computational steps are 
(i) breadth-first searches, (ii) calls to Lemmas~\ref{lem:marked} and~\ref{lem:refinement}, and
(iii)  tests of whether a graph has pathwidth at most $i$ 
for some $i \leq t$ (which can be done in linear time when $t$ is fixed, see~\cite{BK96}), 
and each is performed at most linearly-many times (in fact, only a constant number of times for
(i) and (ii)).  

Since the proof of Theorem~\ref{thm:main} makes at most linearly-many 
calls to Lemmas~\ref{lem:reduction}, \ref{lem:main_pw},  
and~\ref{lem:bounded_pathwidth},   
it remains to show 
that Lemma~\ref{lem:bounded_pathwidth} can be realized in polynomial time. While the proof
itself does not directly yield such an algorithm,  
we note that the problem of finding  
a maximum $\F$-packing and that of finding a minimum $\F$-transversal can both be defined  
in monadic second-order logic. Hence, by Courcelle's Theorem~\cite{C90}, 
both problems can be solved in linear time on graphs of bounded pathwidth, 
and therefore an $\F$-packing and an $\F$-transversal
such as promised by Lemma~\ref{lem:bounded_pathwidth} can be found in linear time. 

Second, a closer inspection shows 
that the running time of the algorithm sketched above 
is not only polynomial for fixed $\F$, but is moreover 
of the form $O(g(\F) \cdot n^{\alpha})$ for $n$-vertex
graphs, where $g$ is a function depending only on $\F$, and $\alpha$ is a
constant independent of $\F$. 
Building on this observation we now sketch a modification of the proof of Theorem~\ref{thm:main} to obtain a {\em single-exponential} 
fixed-parameter tractable (FPT) algorithm for the problem of computing a minimum-size 
$\PP_{t}$-transversal of a graph when parameterized by the size of the optimum, 
where $\PP_{t}$ is the finite set of minimal excluded minors
for the class of graphs with pathwidth strictly less than $t$.  
Here single-exponential FPT means that the running time of the algorithm is 
$O(d^{k} \cdot n^{\alpha})$ on instances $G$ with $\tau_{\PP_{t}}(G) = k$, where $d$ is
a constant depending only on $t$, and as before $\alpha$ is an absolute constant.  
Finding such an algorithm was posed as an open problem by Philip {\it et al.\ }~\cite{PRV10}, 
who gave one for the $t=2$ case. 

Given a pair $(G, k)$,  
our (recursive) FPT algorithm either finds 
a $\PP_{t}$-transversal (abbreviated transversal) of $G$ 
of size at most $k$, or correctly answers that there is no such transversal.  
The algorithm can be briefly described as follows (leaving some details to the reader): 
(I) If $G$ is not connected, let $G_{1}, \dots, G_{p}$ denote its components, and 
for each $i \in [1, p]$ and $\l\in [0, k]$ recurse on $(G_{i}, \l)$. Given the results 
of these recursive calls, 
decide whether $G$ has a transversal of size at most $k$.\footnote{Here we use the fact that all graphs in $\PP_{t}$ are connected, which implies that 
$X$ is a transversal of $G$ if and only if $X \cap V(G_{i})$ is a transversal of $G_{i}$
for each $i\in [1,p]$.} 
(II) If $G$ is connected but has pathwidth strictly less than $t$, return an empty transversal. 
(III) If $G$ is connected with pathwidth at least $t$, 
apply Lemma~\ref{lem:main_pw} on $G$ (with $r$ defined as in the 
proof of Theorem~\ref{thm:main}). 
(III.a) If the outcome of Lemma~\ref{lem:main_pw} 
is a {\tsep{(t-1)}} $(G_{1}, G_{2})$ 
of $G$, apply Lemma~\ref{lem:reduction} to obtain a smaller graph $G'$ and recurse on $(G',k)$; 
if $G'$ has no transversal of size at most $k$ then so does $G$, otherwise 
``lift'' the transversal of $G'$ found back to $G$ and return it.
(III.b) If the outcome of Lemma~\ref{lem:main_pw} is a connected subgraph $G'$ 
of $G$ with pathwdith at least $t$ and at most $f(t,r)$ vertices, {\em branch} on
every non-empty subset of $V(G')$ of size at most $k$, 
namely, for every such subset $Y$, recurse on $(G - Y, k - |Y|)$.  
Observe that every transversal of $G$ contains at least one vertex of $G'$. 
If a transversal of $G-Y$ of size at most $k - |Y|$ is found for some subset $Y$, 
return the union of that transversal with $Y$. 
If no transversal is found for any of the subsets $Y$, 
then $G$ has no transversal of size at most $k$. 

It should be clear that the algorithm finds a transversal of size at most $k$ if there is one. 
The running time is $O(d^{k} \cdot n^{\alpha})$ with $d:=2^{f(t,r)}$ and 
$\alpha$ some absolute constant, since 
we perform at most $d$ recursive calls 
in a branching step, and in each one 
the parameter is decreased by at least one. 

We conclude this section by mentioning that the 
optimization problem of 
finding a minimum $\F$-transversal in a graph has 
been considered for various families $\F$ in several recent
works, see~\cite{CLPPS11, CPPW10, cacti, FLM+11, pumpkins, KPP12, PRV10}.  

Very recently, and independently of this work, Fomin {\it et al.\ }~\cite{FLMS12} 
obtained a single-exponential FPT algorithm for the minimum $\F$-transversal problem 
for every finite set $\F$ of {\em connected} graphs containing at least one planar graph. 
Since all graphs in $\PP_{t}$ are connected, this includes in particular a 
single-exponential FPT algorithm for finding a minimum-size $\PP_{t}$-transversal. 
The authors of~\cite{FLMS12} also presented a {\em randomized} (Monte Carlo) 
constant-factor approximation algorithm for finding a minimum-size $\F$-transversal 
when $\F$ is a finite set of graphs containing at least one planar graph (here the graphs in $\F$
are not assumed to be connected). We note that, 
while our constant-factor approximation algorithm described 
at the beginning of this section is restricted to the case where $\F$ contains a forest,  it is fully deterministic, and provides an $\F$-packing of size within a constant factor of optimal as well. 

\section{An Open Problem}
\label{sec:discussion}

For a finite set $\F$ of graphs containing a forest, let 
$$
\rho(\F) := \sup \frac{\tau_{\F}(G)}{\nu_{\F}(G)}
$$
over all graphs $G$ with $\nu_{\F}(G) > 0$. 
By Theorem~\ref{thm:main} this quantity is well defined. 
When $\F$ consists of a single forest $F$, we simply write $\rho(F)$ for $\rho(\F)$. 
For $t \geq 1$, let 
$\phi(t) := \max\{ \rho(F): F \textrm{ forest}, |F|=t\}$. 

A natural next step would be to investigate the order
of magnitude of $\phi(t)$. 
Our proof gives an upper bound 
which is exponential in $t$, 
and almost certainly far from 
the truth. In particular, it would be interesting to decide whether $\phi(t)$ 
is polynomial in $t$.
As for lower bounds, 
we have that $\rho(F) \geq |F|$ for every forest $F$, 
as can be seen by taking $G$ to be a large complete graph, and hence $\phi(t) \geq t$. 
But we do not know of any super-linear lower bound. 

\section*{Acknowledgments}
We would like to acknowledge the participation of 
Bundit Laekhanukit and Guyslain Naves in the early stages of this research project. 
We thank them for several stimulating discussions. 
We also thank Dimitrios Thilikos for useful discussions 
regarding folios. Finally, we thank the anonymous referee for her/his thorough reading and very helpful comments. 

\bibliographystyle{plain}
\bibliography{forest_minors}

\end{document}